\documentclass{amsart}

\usepackage[utf8]{inputenc}
\usepackage{amsmath,amsfonts,amsthm,amsopn,color,amssymb,enumitem}
\usepackage{graphicx}
\usepackage{caption}
\usepackage{subcaption}
\usepackage[colorlinks=true, urlcolor=blue, citecolor=red, linkcolor=blue]{hyperref}
\usepackage{cite}
\usepackage{esint}
\usepackage{pgfplots}
\usepackage{mathrsfs}
\usetikzlibrary{arrows}
\usepackage{verbatim}
\usepackage{mathtools}
\usepackage[margin=2.4cm]{geometry}

\pgfplotsset{compat=1.18}

\newtheorem{theorem}{Theorem}[section]
\newtheorem{lemma}[theorem]{Lemma}

\newtheorem{proposition}[theorem]{Proposition}

\newcommand{\R}{\mathbb{R}}

\renewcommand{\le}{\leqslant}

\def\bbm[#1]{\mbox{\boldmath $#1$}}
\newcommand{\beq }{\begin{equation}}
\newcommand{\eeq }{\end{equation}}

\begin{document}

\title[Regularity for widely degenerate elliptic equations]{On the second-order regularity of widely degenerate elliptic systems}

\address{Dipartimento di Matematica e Informatica, Università della Calabria, Ponte Pietro Bucci 31B, 87036 Arcavacata di Rende, Cosenza, Italy}
\author[F. Iandoli]{Felice Iandoli}
\email{felice.iandoli@unical.it}

\author[B. Sciunzi]{Berardino Sciunzi}
\email{berardino.sciunzi@unical.it}

\author[D. Vuono]{Domenico Vuono}
\email{domenico.vuono@unical.it}

\thanks{The authors are supported by PRIN PNRR P2022YFAJH \emph{Linear and Nonlinear PDEs: New directions and applications} and by Gruppo Nazionale per l'Analisi Matematica, la Probabilità. B. Sciunzi and D. Vuono have been partially supported by \emph{INdAM-GNAMPA Esistenza, regolarità e proprietà qualitative per problemi non lineari} E5324001950001. F. Iandoli is supported by \emph{INdAM-GNAMPA Esistenza e proprietà qualitative per soluzioni di equazioni ellittiche con termini di ordine inferiore} E5324001950001.}


\begin{abstract}
We prove local second-order regularity near the critical set for widely degenerate elliptic systems, and show that this degenerate region has zero Lebesgue measure provided the source term is non-vanishing.
\end{abstract}

\maketitle

\section{Introduction}

In this work, we consider weak solutions to the following widely degenerate elliptic system:
\begin{equation}\label{eq:Equazione_Principale}
    -\operatorname{\bf div} \left((|D \boldsymbol{u}|-1)_+^{p-1}\frac{D \boldsymbol{u}}{|D \boldsymbol{u}|}\right)=\boldsymbol{f}(x)\quad \text{in }\Omega,
\end{equation}
where $p>1$, $D{\boldsymbol u}$ is the Jacobian of the vector field ${\boldsymbol u}\,:\, \Omega \rightarrow \mathbb{R}^N$, with $N\geq 1$, $\Omega$ is an open set of $\mathbb{R}^n$, with $n\geq 2$, and $(\cdot)_+$ stands for the positive part. We shall use the notation ${\boldsymbol u}=(u^1,\ldots,u^N)$ and ${\boldsymbol f}=(f^1,\ldots,f^N)$.

Equations of the type \eqref{eq:Equazione_Principale} fall into the class of \emph{widely degenerate} elliptic problems. The main feature of this operator is that its ellipticity vanishes on the whole zone where $\{|D\boldsymbol{u}| \leq 1\}$. Such operators naturally arise in optimal transport problems and congested traffic dynamics, where the threshold $|D\boldsymbol{u}| = 1$ separates congested and uncongested regimes. We briefly recall the state of the art. For instance, Lipschitz regularity of minimizers has 
been established in \cite{Brasco,BrascoAmbrogio,Passarelli}. In 
\cite{Sant,CF1,CF2,BDGP1,BDGP2} it is proved that the composition of any continuous 
function vanishing on the set $\{|D\mathbf{u}|\le 1\}$ with $D\mathbf{u}$ is continuous. 
For second–order estimates in the scalar case we refer to \cite{AGP,GR}.

\noindent Given two matrices $\mathcal{M},\mathcal{N} \in \mathbb{R}^{q \times r}$ we define the scalar product
$ 
\mathcal{M}:\mathcal{N}=\sum_{i=1}^{q} \langle\mathcal{M}^i ,\mathcal{N}^i\rangle=\sum_{i=1}^{q} \sum_{j=1}^{r} \mathcal{M}^i_j \mathcal{N}^i_j,
$ 
 whereas $ \langle\cdot,\cdot\rangle$ denotes the scalar product of two real vectors. Moreover, we shall use the so called \emph{infinity Laplacian} $$\Delta_{\infty} \boldsymbol{u}:=D\boldsymbol{u} (D^{2}\boldsymbol{u}\, D\boldsymbol{u}),$$ see \eqref{eq:vect}, whose $k$-th component is given by $\sum_{i,s=1}^{n}\sum_{l=1}^{N} \partial_{is}u^{\,l}\,\partial_{i}u^{\,k}\, \partial_{s}u^{\,l}$, for $k= 1,\ldots,N$. For further details on this notation, we refer to Section~\ref{seconda}.

We say that a vector field $\boldsymbol{u}$ is a weak solution to the problem \eqref{eq:Equazione_Principale} if and only if $\boldsymbol{u} \in W^{1,p}_{loc}(\Omega):=W^{1,p}_{loc}(\Omega;\mathbb{R}^N)$ and 
\begin{equation}\label{weakN}
\int_{\Omega} (|D \boldsymbol{u}|-1)_+^{p-1}\left(\frac{D \boldsymbol{u}}{|D \boldsymbol{u}|}: D {\boldsymbol {\psi}}\right)\, dx = \int_{\Omega} \langle{\boldsymbol f}, {\boldsymbol{\psi}}\rangle\, dx, \quad \forall  {\boldsymbol {\psi}} \in C^{\infty}_c(\Omega).
\end{equation}

Our first  result concerns the second-order estimates for weak solutions of \eqref{eq:Equazione_Principale}. 

\begin{theorem}\label{TeoDerivateSeconde}
 Let $\Omega$ be an open set of $\R^n$ and  $\boldsymbol{f}\in W^{1,1}_{loc}(\Omega;\R^N)\cap L^q_{loc}(\Omega;\R^N)$, with $q>n$. Consider $\boldsymbol{u}$ a weak solution of \eqref{eq:Equazione_Principale}. For any  $\alpha<\min\{p-1,1\}$ and  for any $\tilde \Omega\subset\subset \hat \Omega\subset\subset\Omega$ open sets, we have 
  \begin{equation}\label{eq:stimaseconde}
     \int_{\tilde \Omega\cap\{|D\boldsymbol{u}|>1\}}(|{D\boldsymbol{u}}|-1)_+^{p-2-\alpha}\left|\Delta_{\infty} \boldsymbol{u} \right|^2\,dx\leq C,
 \end{equation}
 where $C=C(p,n,N,\alpha,\hat\Omega,\|\boldsymbol{f}\|_{L^q(\hat \Omega)},\|\boldsymbol{f}\|_{W^{1,1}(\hat \Omega)},\|D\boldsymbol{u}\|_{L^p(\hat \Omega)})$ is a positive constant. 
 \end{theorem}

Our second result establishes a quantitative control on the integrability of the inverse of the degenerate weight, provided the source term is bounded away from zero along at least one direction. This result is obtained, in the spirit of \cite{DS}, thanks to Theorem \ref{TeoDerivateSeconde}, and improves the result in \cite{GR}.

\begin{theorem}\label{TeoinversodelPeso}
   Under the same assumptions of Theorem \ref{TeoDerivateSeconde}, suppose additionally that for some index $k \in \{1, \dots, N\}$ we have $f^k\geq \tau>0$ in $\hat\Omega$. Then, for any $  \beta <\min\{2p-3,p-1\},$ 
the following estimate holds true 
\begin{equation}\label{eq:stimainverso}
    \int_{\tilde\Omega} \frac{1}{(|D\boldsymbol{u}|-1)_+^\beta}\,dx \leq C,
\end{equation}
where $C=C(p,n,N,\beta,\hat \Omega,\|\boldsymbol{f}\|_{L^q(\hat \Omega)},\|\boldsymbol{f}\|_{W^{1,1}(\hat \Omega)},\|D\boldsymbol{u}\|_{L^p(\hat \Omega)},\tau)$ is a positive constant. In particular, denoting with $\mathcal{L}$ the Lebesgue measure on $\R^n$, we have $$\mathcal{L}(\{x \in \tilde{\Omega} : |D\boldsymbol{u}(x)|\leq 1\})=0.$$
\end{theorem}

As a consequence of Theorems~\ref{TeoDerivateSeconde} and~\ref{TeoinversodelPeso}, we obtain the following result.

\begin{theorem}\label{Corollario2}
   Under the same assumptions of Theorem \ref{TeoDerivateSeconde}, if 
 $1<p < 3$, we have that \begin{equation}\label{corollario}
     \Delta_{\infty} \boldsymbol{u} \in L^2(\tilde \Omega\cap\{|D\boldsymbol{u}|>1\}).
 \end{equation}
If $p\geq 3$, under the assumptions of Theorem \ref{TeoinversodelPeso}, the following holds true $$\Delta_{\infty} \boldsymbol{u}\in L^q(\tilde \Omega)\quad\text{for any }  1\leq q < \frac{p-1}{p-2}.$$
\end{theorem}

\section{Notation and preliminary results}\label{seconda}
\subsection*{Notation}\label{Notazioni}
Generic fixed numerical constants will be denoted by $C$ and will be allowed to vary within a single line or formula. 

We will use the bold style to stress the vectorial nature of  different quantities. For instance, a $N$-vectorial function $\boldsymbol{u}$ defined in $\Omega$ will be written as
$
\boldsymbol{u}(x)=\left(u^1(x),\ldots,u^N(x) \right),
$
where $u^\ell$ is a scalar function defined in $\Omega$ for $\ell=1,\ldots,N$. We denote by \(D\boldsymbol{u}\) the Jacobian matrix of \(\boldsymbol{u}\):
\begin{equation*}
D {\boldsymbol u}=  
        ( \nabla u^1 
         \cdots 
         \nabla u^N)^T
        \in \R^{N\times n}, \quad \nabla u^\ell = \left(\frac{\partial u^\ell}{\partial x_1}\ldots\frac{\partial u^\ell}{\partial x_n}\right)
        \end{equation*}
for $\ell=1,\ldots,N$, and
\[\displaystyle{|D {\boldsymbol u}|=\sqrt{\sum_{\ell=1}^N\sum_{j=1}^n \left(\frac{\partial u^\ell}{\partial x_j}\right)^2}}.\] 
We also use the notations $ u_j^\ell=\dfrac{\partial u^\ell}{\partial x_j}$ and $ u_{ij}^\ell=\dfrac{\partial^2 u^\ell}{\partial x_i\partial x_j}$. In particular we denote with $\boldsymbol{u}_i:=\left(u_i^1,...,u_i^N\right)$.

Let $\{ \boldsymbol{e}^\alpha\}_{\alpha =1}^N$, and ${\{ \boldsymbol{e}_j\}}^n_{j =1}$ be the canonical basis of $\mathbb{R}^N$ and $\mathbb{R}^n$, respectively. We will also denote a second-order tensor of size $N\times n$ as
$\eta = \eta_j^\alpha \ \boldsymbol{e}^\alpha \otimes \boldsymbol{e}_j$,
and a third-order tensor of size $N\times n \times n$ as
$    \xi = \xi_{ij}^\alpha \ \boldsymbol{e}^\alpha \otimes \boldsymbol{e}_j \otimes \boldsymbol{e}_i$.
In particular, by $D^2\boldsymbol{u}$ we denote 
 $D^2 \boldsymbol{u} = u_{ij}^\alpha \ \boldsymbol{e}^\alpha \otimes \boldsymbol{e}_j \otimes\boldsymbol{e}_i$,
and we define its norm as follows
\[\|D^2 {\boldsymbol u} \|=\sqrt{\sum_{\alpha=1}^N\sum_{i,j}(u^{\alpha}_{ij})^2}.\]

We now define the vector $D^2\boldsymbol{u}D\boldsymbol{u}$ as the application of a third-order tensor to a second-order one, namely
\begin{equation}\label{eq:vect}
D^2\boldsymbol{u}D\boldsymbol{u}:=\left(D\boldsymbol{u}:D\boldsymbol{u}_{1}\,\,
\cdots\,\, D\boldsymbol{u}:D\boldsymbol{u}_{i}\,\,\cdots\,\,
D\boldsymbol{u}:D\boldsymbol{u}_{n}
\right)^T,
\end{equation}
where $\boldsymbol{u}_i=\partial_{x_i}\boldsymbol{u}$ and \(D\boldsymbol{u}_i \in \mathbb{R}^{N \times n}\) denotes the following matrix

\[
D\boldsymbol{u}_i :=\left(\nabla u^1_i
\,\,\cdots 
\nabla u^N_i\right)^T,
\]
for a fixed \(i \in \{1,\dots,n\}\). In the following we shall often use the inequality
\begin{equation}\label{eq:FraCo}
|D^2\boldsymbol u D\boldsymbol u|\leq |D\boldsymbol{u}|\|D^2\boldsymbol{u}\|.
\end{equation}

\subsection{Regularized problem}
Here and in the following we denote by \(\boldsymbol{u} \in W^{1,p}_{loc}(\Omega)\) a weak solution of \eqref{eq:Equazione_Principale}. By \cite{Passarelli}, we know that weak solutions of \eqref{eq:Equazione_Principale} belong to  \(W^{1,\infty}_{\mathrm{loc}}(\Omega)\), see \cite{Brasco,BrascoAmbrogio} for the scalar case. As a consequence, we may always assume that \(D\boldsymbol{u}\) is locally bounded in \(\Omega\).

At this stage we introduce the regularised problem and describe several of its properties, referring to \cite{BDGP1,BDGP2} for a more comprehensive discussion. For every \(\varepsilon \in (0,1]\), we introduce the  function
$$h_\varepsilon(t) := \frac{(t-1)_+^{p-1}}{t} + \varepsilon, \qquad t \in [0,\infty).$$
 Note that \(h_\varepsilon\) never vanishes, since \(h_\varepsilon(t) \equiv \varepsilon\) on the whole interval \([0,1]\). Concerning its regularity, we observe that if \(p>2\) then \(h_\varepsilon \in C^{1}([0,\infty))\) whereas if \(p \leq 2\) we have \(h_\varepsilon \in W^{1,1}((0,\infty) \cap C^{1}([0,1)) \cap (1,\infty))\). In the borderline case \(p=2\) the regularity improves to \(h_\varepsilon \in W^{1,\infty}([0,\infty))\). Based on this regularisation, we define the associated vector field
$$\mathcal{A}_\varepsilon(\xi) := h_\varepsilon(|\xi|)\,\xi, \qquad \xi \in \mathbb{R}^{Nn}.$$

Let $x_0\in \Omega$ and  $B_{R}:=B_{R}(x_0)\subset\subset\Omega$, where $B_{R}(x_0)$ denotes the open ball of radius $R$ centered at $x_0$. By \(\boldsymbol{u}_\varepsilon -\boldsymbol{u}\in   W^{1,\mathfrak{p}}_{0}(B_R;\mathbb{R}^N)\), where \(\mathfrak{p} = \max\{2,p\}\), we denote the unique weak solution of the regularized elliptic system.
\begin{equation} \label{eq:problregol}
		\begin{cases}
			-\operatorname{\bf div}\left( \mathcal{A}_\varepsilon(D\boldsymbol{u}_\varepsilon) \right) = \boldsymbol{f}(x) & \text{in } B_{R} \\
			\boldsymbol{u}_\varepsilon = \boldsymbol{u} & \text{on } \partial B_{R}.
\end{cases}
	\end{equation}
	 
The weak formulation of the previous problem is \begin{equation}\label{eq:deboleregolar}
    \begin{split}
        &\int_{B_{R}} \varepsilon (D\boldsymbol{u}_\varepsilon:D\boldsymbol{\zeta})+(|D\boldsymbol{u}_\varepsilon|-1)_+^{p-1}\left(\frac{D\boldsymbol{u}_\varepsilon}{|D\boldsymbol{u}_\varepsilon|}:D\boldsymbol{\zeta} \right)\,dx\\&=\int_{B_{R}} \langle \boldsymbol{f},\boldsymbol{\zeta}\rangle \,dx,\quad \forall  \boldsymbol{\zeta} \in C^\infty_c(B_{R}).
\end{split}
    \end{equation}
    
Note that \(\boldsymbol{u}_\varepsilon \in W^{1,\infty}_{\mathrm{loc}}(B_R;\mathbb{R}^N) \cap W^{2,2}_{\mathrm{loc}}(B_R;\mathbb{R}^N)\). This follows by arguments analogous to those in \cite{BDGP1,Marcello}. This regularity will allow us to carry out the computations appearing in the subsequent sections.

\section{Second-order estimates}

This section is devoted to the proof of Theorem \ref{TeoDerivateSeconde}.
To begin with, we introduce the linearized equation. Note that, when \(p<2\), the operator becomes singular in the region \(|D\boldsymbol{u}_\varepsilon|\leq 1\). Below (see \eqref{testphi}), we shall introduce a suitable cut-off function to stay well separated from the set \(\{\,|D\boldsymbol{u}_\varepsilon|\leq 1\,\}\). In order to do so, we note that the set \(\{\,|D\boldsymbol{u}_\varepsilon|\leq 1\,\}\) is closed, since the function \(|D\boldsymbol{u}_\varepsilon|\) is Hölder continuous in the set \(\{\,|D\boldsymbol{u}_\varepsilon|> 1\,\}\), if $\boldsymbol{f}\in L^q(B_{R})$, with $q>n$; see \cite[Theorem~3.6]{BDGP1}. With this in mind, we consider the following function $\boldsymbol{\zeta}:=\boldsymbol{\varphi}_i,$ with $\boldsymbol\varphi\in C_c^\infty\big(B_{R}\setminus\{|D\boldsymbol u_\varepsilon|\leq 1\}\big)$ and $i\in \{1,...,n\}$ fixed, and we test the equation \eqref{eq:deboleregolar} with $\boldsymbol{\zeta}$, obtaining
\begin{equation}\label{primoregolarizzato}
    \begin{split}
        &\int_{B_{R}} \varepsilon (D\boldsymbol{u}_{\varepsilon,i}:D\boldsymbol{\varphi})+\int_{B_{R}}(p-1)(|{D\boldsymbol{u}_\varepsilon}|-1)_+^{p-2}\left({D\boldsymbol{u}_{\varepsilon,i}}:\frac{{D\boldsymbol{u}_\varepsilon}}{|{D\boldsymbol{u}_\varepsilon}|} \right)\left(\frac{{D\boldsymbol{u}_\varepsilon}}{|{D\boldsymbol{u}_\varepsilon}|}:{D\boldsymbol \varphi} \right)\,dx
        \\&+\int_{B_{R}}(|{D\boldsymbol{u}_\varepsilon}|-1)_+^{p-1}\left(\frac{{D\boldsymbol{u}_{\varepsilon,i}}}{|{D\boldsymbol{u}_\varepsilon}|}:{D\boldsymbol \varphi} \right)\,dx
        \\&-\int_{B_{R}}(|{D\boldsymbol{u}_\varepsilon}|-1)_+^{p-1}\left(\frac{{D\boldsymbol{u}_{\varepsilon,i}}}{|{D\boldsymbol{u}_\varepsilon}|}:\frac{{D\boldsymbol{u}_\varepsilon}}{|{D\boldsymbol{u}_\varepsilon}|} \right)\left(\frac{{D\boldsymbol{u}_\varepsilon}}{|{D\boldsymbol{u}_\varepsilon}|}:{D\boldsymbol \varphi} \right)\,dx\\&=\int_{B_{R}} \langle \boldsymbol{f}_{i},\boldsymbol{\varphi}\rangle \,dx,\quad \forall  \boldsymbol{\varphi} \in C^\infty_c\big(B_{R}\setminus\{|D\boldsymbol u_\varepsilon|\leq 1\}\big).
\end{split}
    \end{equation}

Now we introduce some auxiliary functions that will be useful in the sequel. For $0<2\tilde R<R$, we consider a cut-off function ${\psi}_{\tilde R} := {\psi} \in C^{\infty}_{c}(B_{2\tilde R}(x_0))$ such that:
    \begin{equation} \label{eq:varphi}
        \begin{cases}
        {\psi} = 1 & \text{in} \quad B_{\tilde R}(x_0)\\
        |\nabla\psi| \leq \frac{2}{\tilde R} & \text{in} \quad B_{2\tilde R}(x_0)\setminus B_{\tilde R}(x_0).\\
    \end{cases}
\end{equation}
 
 Now, for $\alpha\in\mathbb{R}$ and $\delta>0$,  let us define

 \begin{equation}\label{testphi}
     \boldsymbol{\varphi}:= \frac{(|D\boldsymbol u_\varepsilon|-1-\delta)_+}{(|D\boldsymbol u_\varepsilon|-1)_+^\alpha}{{\boldsymbol u}_{\varepsilon,i}}{\psi}^2.
\end{equation}
We note that \(\boldsymbol{\varphi}\in W^{1,2}_0(B_{R})\) is an admissible test function, see \cite[Chapter 5]{Evans}, and by a standard density argument it can be plugged in \eqref{primoregolarizzato}. In particular, we have \begin{equation}\label{gradientebombolo}
\begin{split}
     {D\boldsymbol \varphi}&=\left(1-\alpha\frac{(|D\boldsymbol u_\varepsilon|-1-\delta)_+}{(|D\boldsymbol u_\varepsilon|-1)_+}\right)\frac{1}{(|D\boldsymbol u_\varepsilon|-1)_+^\alpha}{\psi}^2\frac{{{\boldsymbol u}_{\varepsilon,i}}\otimes{D^2\boldsymbol{u}_\varepsilon}{D\boldsymbol{u}_\varepsilon}}{|{D\boldsymbol{u}_\varepsilon}|}\chi_{\delta,\varepsilon}
     \\&+\frac{(|D\boldsymbol u_\varepsilon|-1-\delta)_+}{(|D\boldsymbol u_\varepsilon|-1)_+^\alpha}{\psi}^2{D\boldsymbol{u}_{\varepsilon,i}}+2\frac{(|D\boldsymbol u_\varepsilon|-1-\delta)_+}{(|D\boldsymbol u_\varepsilon|-1)_+^\alpha}{\psi} ({{\boldsymbol u}_{\varepsilon,i}}\otimes\nabla\psi),
     \end{split}
\end{equation}
where \(\chi_{\delta,\varepsilon}\) denotes the characteristic function of the set \(\{\,|D\boldsymbol{u}_\varepsilon| > 1+\delta\,\}\), for any $\delta \geq 0$. 

\begin{lemma}\label{Lemma11}
   Let \(\Omega\) be an open set in \(\mathbb{R}^n\) and assume that \(B_{R}\subset\subset \Omega\). Let \(\boldsymbol{u}_\varepsilon\) be a weak solution of \eqref{eq:problregol} and let \(\boldsymbol{\varphi}\) be the function defined in \eqref{testphi}. Then, for $0<\alpha<1$, the following inequality holds 

\begin{equation}\label{Stimasulprimoregolarizzato}
    \begin{split}
          &\int_{B_{R}} {(|D\boldsymbol u_\varepsilon|-1)_+^{-\alpha}}\left|\frac{{D^2\boldsymbol{u}_\varepsilon}{D\boldsymbol{u}_\varepsilon}}{|{D\boldsymbol{u}_\varepsilon}|}\right|^2|{D\boldsymbol{u}_\varepsilon}|{\psi}^2\chi_{0,\varepsilon}\,dx
          +\int_{B_{R}} {(|D\boldsymbol u_\varepsilon|-1)_+^{1-\alpha}}\|{D^2\boldsymbol{u}_\varepsilon}\|^2{\psi}^2\,dx
         \\&\leq C(\alpha)\int_{B_{R}} {(|D\boldsymbol u_\varepsilon|-1)_+^{1-\alpha}}|{D\boldsymbol{u}_\varepsilon}|^2|\nabla\psi|^2\,dx
         +C(\alpha)\sum_{i=1}^n\int_{B_{R}}  (D\boldsymbol{u}_{\varepsilon,i}:D\boldsymbol{\varphi})\,dx,
    \end{split}
\end{equation}
where $C(\alpha)$ is a positive constant.
\end{lemma}

\begin{proof}
 Recalling \eqref{gradientebombolo}, we set
\begin{equation}\label{gradientetestphi}
\begin{split}
     {D\boldsymbol \varphi}:=\mathcal{J}_1+\mathcal{J}_2+\mathcal{J}_3,
     \end{split}
\end{equation}

which yields

\begin{equation}\label{robe}
    \int_{B_{R}}  (D\boldsymbol{u}_{\varepsilon,i}:D\boldsymbol{\varphi})=\int_{B_{R}}  (D\boldsymbol{u}_{\varepsilon,i}:\mathcal{J}_1+\mathcal{J}_2+\mathcal{J}_3).
\end{equation}
We now turn to the estimate of the right-hand side of \eqref{robe}. More precisely, the contributions coming from \({\mathcal{J}}_1\) and \({\mathcal{J}}_2\) will be bounded from below, whereas those arising from \({\mathcal{J}}_3\) will be bounded from above.

Considering only the terms involving $\mathcal{J}_1$ in the right-hand side of \eqref{robe}, we have
\begin{equation*}
    \begin{split}
        &\int_{B_{R}} \left(1-\alpha\frac{(|D\boldsymbol u_\varepsilon|-1-\delta)_+}{(|D\boldsymbol u_\varepsilon|-1)_+}\right)\frac{1}{(|D\boldsymbol u_\varepsilon|-1)_+^\alpha}\left({{D\boldsymbol{u}_{\varepsilon,i}}}:\frac{{\boldsymbol u}_{\varepsilon,i}\otimes{D^2\boldsymbol{u}_\varepsilon}{D\boldsymbol{u}_\varepsilon}}{|{D\boldsymbol{u}_\varepsilon}|}\right){\psi}^2\chi_{\delta,\varepsilon}\,dx.
\end{split}
\end{equation*}
 Recalling \eqref{eq:vect} and denoting by \((D^{2}\boldsymbol{u}_\varepsilon\, D\boldsymbol{u}_\varepsilon)_k\) the \(k\)-th component of the corresponding vector, taking the sum over \(i=1,\ldots,n\) and observing that
\begin{equation}\label{stima1}
    \begin{split}
        &\sum_i \left({{D\boldsymbol{u}_{\varepsilon,i}}}:{\boldsymbol u}_{\varepsilon,i}\otimes{D^2\boldsymbol{u}_\varepsilon}{D\boldsymbol{u}_\varepsilon}\right)=\sum_{i,k}\sum_l{u}^l_{\varepsilon,ki}{u}^l_{\varepsilon,i}({D^2\boldsymbol{u}_\varepsilon}{D\boldsymbol{u}_\varepsilon})_k
        \\&=\sum_k ({D^2\boldsymbol{u}_\varepsilon}{D\boldsymbol{u}_\varepsilon})_k^2=|{D^2\boldsymbol{u}_\varepsilon}{D\boldsymbol{u}_\varepsilon}|^2,
    \end{split}
\end{equation} 
we obtain 
\begin{equation*}
    \begin{split}
        &\int_{B_{R}} \left(1-\alpha\frac{(|D\boldsymbol u_\varepsilon|-1-\delta)_+}{(|D\boldsymbol u_\varepsilon|-1)_+}\right)\frac{1}{(|D\boldsymbol u_\varepsilon|-1)_+^\alpha}\left|\frac{{D^2\boldsymbol{u}_\varepsilon}{D\boldsymbol{u}_\varepsilon}}{|{D\boldsymbol{u}_\varepsilon}|}\right|^2|{D\boldsymbol{u}_\varepsilon}|{\psi}^2\chi_{\delta,\varepsilon}\,dx.
\end{split}
\end{equation*}
Since $(|D\boldsymbol u_\varepsilon|-1-\delta)_+\leq (|D\boldsymbol u_\varepsilon|-1)_+$ and $\alpha>0$, the contribution coming from $\mathcal J_1$ is estimated from below by
\begin{equation}\label{primotermine}
    \begin{split}
        &(1-\alpha)\int_{B_{R}} {(|D\boldsymbol u_\varepsilon|-1)_+^{-\alpha}}\left|\frac{{D^2\boldsymbol{u}_\varepsilon}{D\boldsymbol{u}_\varepsilon}}{|{D\boldsymbol{u}_\varepsilon}|}\right|^2|{D\boldsymbol{u}_\varepsilon}|{\psi}^2\chi_{\delta,\varepsilon}\,dx
        \end{split}
\end{equation}

The term involving $\mathcal{J}_2$ in \eqref{robe} equals
\begin{equation*}\label{intermedia}
    \begin{split}
         &\int_{B_{R}} {\frac{(|D\boldsymbol u_\varepsilon|-1-\delta)_+}{(|D\boldsymbol u_\varepsilon|-1)_+^\alpha}}\left({D\boldsymbol{u}_{\varepsilon,i}}:D{{\boldsymbol u}_{\varepsilon,i}}\right){\psi}^2\,dx.
\end{split}
\end{equation*} 

Taking the sum with respect to $i = 1,\ldots,n$, we obtain
\begin{equation}\label{secondotermine}
    \begin{split}
         &\int_{B_{R}} {\frac{(|D\boldsymbol u_\varepsilon|-1-\delta)_+}{(|D\boldsymbol u_\varepsilon|-1)_+^\alpha}}\left\|{D^2\boldsymbol u_\varepsilon}\right\|^2{\psi}^2\,dx.
\end{split}
\end{equation}
 
We conclude with the terms involving $\mathcal{J}_3$. In particular, after summing over $i=1,...,n$, and noting that 
\begin{equation}\label{cosebelle}
\begin{split}
&\sum_i ({D\boldsymbol{u}_{\varepsilon,i}}:({{\boldsymbol u}_{\varepsilon,i}}\otimes\nabla\psi))=\sum_{i,k}\sum_lu^l_{\varepsilon,ik}u^l_{\varepsilon,i}\psi_k
\\&=\langle D^2\boldsymbol{u}_\varepsilon D\boldsymbol{u}_\varepsilon,\nabla \psi\rangle\leq |D^2\boldsymbol{u}_\varepsilon D\boldsymbol{u}_\varepsilon||\nabla \psi|,
    \end{split}
\end{equation}
we infer that 
\begin{equation*}
    \begin{split}
\int_{B_{R}}\left(D\boldsymbol{u}_{\varepsilon,i}:\mathcal{J}_3\right)\,dx&= 2\int_{B_{R}} {\frac{(|D\boldsymbol u_\varepsilon|-1-\delta)_+}{(|D\boldsymbol u_\varepsilon|-1)_+^\alpha}}({D\boldsymbol{u}_{\varepsilon,i}}:({{\boldsymbol u}_{\varepsilon,i}}\otimes\nabla\psi)){\psi}\,dx
\\&\leq  2\int_{B_{R}} {\frac{(|D\boldsymbol u_\varepsilon|-1-\delta)_+}{(|D\boldsymbol u_\varepsilon|-1)_+^\alpha}}\left|{D^2\boldsymbol u_\varepsilon} {D\boldsymbol{u}_\varepsilon}\right||\nabla\psi|{\psi}\,dx:=\hat{\mathcal{I}}.
\end{split}
\end{equation*}
By weighted Young's inequality with exponents $(1/2,1/2)$, we have
\begin{equation}\label{termineJ_4}
    \begin{split}
         \hat{\mathcal{I}}&\leq\hat \varepsilon \int_{B_{R}} {\frac{(|D\boldsymbol u_\varepsilon|-1-\delta)_+}{(|D\boldsymbol u_\varepsilon|-1)_+^\alpha}}\|{D^2\boldsymbol{u}_\varepsilon}\|^2{\psi}^2\,dx
         +C(\hat \varepsilon)\int_{B_{R}} {\frac{(|D\boldsymbol u_\varepsilon|-1-\delta)_+}{(|D\boldsymbol u_\varepsilon|-1)_+^\alpha}}|{D\boldsymbol{u}_\varepsilon}|^2|\nabla\psi|^2\,dx.
\end{split}
\end{equation}

Now collecting \eqref{primotermine}, \eqref{secondotermine},  \eqref{termineJ_4}, and, by plugging these into \eqref{robe}, we get
\begin{equation*}
    \begin{split}
         (1-\alpha) &\int_{B_{R}} {(|{D\boldsymbol{u}_\varepsilon}|-1)_{+}^{-\alpha}}\left|\frac{{D^2\boldsymbol{u}_\varepsilon}{D\boldsymbol{u}_\varepsilon}}{|{D\boldsymbol{u}_\varepsilon}|}\right|^2|{D\boldsymbol{u}_\varepsilon}|{\psi}^2\chi_{\delta,\varepsilon}\,dx
         +(1-\hat\varepsilon)\int_{B_{R}} {\frac{(|D\boldsymbol u_\varepsilon|-1-\delta)_+}{(|D\boldsymbol u_\varepsilon|-1)_+^\alpha}}\|{D^2\boldsymbol{u}_\varepsilon}\|^2{\psi}^2\,dx
         \\&\leq C(\hat \varepsilon)\int_{B_{R}} {\frac{(|D\boldsymbol u_\varepsilon|-1-\delta)_+}{(|D\boldsymbol u_\varepsilon|-1)_+^\alpha}}|{D\boldsymbol{u}_\varepsilon}|^2|\nabla\psi|^2\,dx
         +\sum_{i=1}^n\int_{B_{R}}  (D\boldsymbol{u}_{\varepsilon,i}:D\boldsymbol{\varphi})\,dx.
\end{split}
\end{equation*}
Using the fact that \((|D\boldsymbol{u}_\varepsilon|-1-\delta)_+ \leq (|D\boldsymbol{u}_\varepsilon|-1)_+\) in the right-hand side of the previous inequality, we get
\begin{equation*}
    \begin{split}
         (1-\alpha) &\int_{B_{R}} {(|{D\boldsymbol{u}_\varepsilon}|-1)_+^{-\alpha}}\left|\frac{{D^2\boldsymbol{u}_\varepsilon}{D\boldsymbol{u}_\varepsilon}}{|{D\boldsymbol{u}_\varepsilon}|}\right|^2|{D\boldsymbol{u}_\varepsilon}|{\psi}^2\chi_{\delta,\varepsilon}\,dx
         +(1-\hat\varepsilon)\int_{B_{R}} {\frac{(|D\boldsymbol u_\varepsilon|-1-\delta)_+}{(|D\boldsymbol u_\varepsilon|-1)_+^\alpha}}\|{D^2\boldsymbol{u}_\varepsilon}\|^2{\psi}^2\,dx
         \\&\leq C(\hat \varepsilon)\int_{B_{R}} {(|D\boldsymbol u_\varepsilon|-1)_+^{1-\alpha}}|{D\boldsymbol{u}_\varepsilon}|^2|\nabla\psi|^2\,dx
         +\sum_{i=1}^n\int_{B_{R}}  (D\boldsymbol{u}_{\varepsilon,i}:D\boldsymbol{\varphi})\,dx.
\end{split}
\end{equation*}
Letting \(\delta \to 0\), we obtain the desired conclusion.

\end{proof}

\begin{lemma}\label{Lemma2}
  Let \(\Omega\) be an open set in \(\mathbb{R}^n\) and assume that \(B_{R}\subset\subset \Omega\). Let \(\boldsymbol{u}_\varepsilon\) be a weak solution of \eqref{eq:problregol} and let \(\boldsymbol{\varphi}\) be the function defined in \eqref{testphi}. Then, for $0<\alpha<\min\{p-1,1\}$, the following inequality holds
\begin{equation}\label{Lemma22}
    \begin{split}
    &\int_{B_{R}}(|{D\boldsymbol{u}_\varepsilon}|-1)_+^{p-2-\alpha}\left| \frac{\Delta_{\infty}\boldsymbol{u}_\varepsilon}{|{D\boldsymbol{u}_\varepsilon}|^2} \right|^2|{D\boldsymbol{u}_\varepsilon}|{\psi}^2\chi_{0,\varepsilon}\,dx
    \\&+\int_{B_{R}}(|{D\boldsymbol{u}_\varepsilon}|-1)_+^{p-1-\alpha}\left|\frac{{D^2\boldsymbol u_\varepsilon}{D\boldsymbol{u}_\varepsilon}}{|{D\boldsymbol{u}_\varepsilon}|} \right|^2{\psi}^2\,dx
    \\&\leq C(p,\alpha)\int_{B_{R}}(|{D\boldsymbol{u}_\varepsilon}|-1)_+^{p-1-\alpha}|{D\boldsymbol{u}_\varepsilon}|^2|\nabla\psi|^2\,dx
\\&+C(p,\alpha)\sum_{i=1}^n\int_{B_{R}}\partial_{x_i}\left((|D\boldsymbol{u}_\varepsilon|-1)_+^{p-1}\frac{D\boldsymbol{u}_\varepsilon}{|D\boldsymbol{u}_\varepsilon|}\right):D\boldsymbol{\varphi} \,dx
    \end{split}
\end{equation}
where $C(p,\alpha)$ is a positive constant.
\end{lemma}

\begin{proof}
  We begin the proof observing that the following equality holds
\begin{equation}\label{cose2}
    \begin{split}
        &\int_{B_{R}}\partial_{x_i}\left((|D\boldsymbol{u}_\varepsilon|-1)_+^{p-1}\frac{D\boldsymbol{u}_\varepsilon}{|D\boldsymbol{u}_\varepsilon|}\right):D\boldsymbol{\varphi} \,dx\\&=\int_{B_{R}}(p-1)(|{D\boldsymbol{u}_\varepsilon}|-1)_+^{p-2}\left({D\boldsymbol{u}_{\varepsilon,i}}:\frac{{D\boldsymbol{u}_\varepsilon}}{|{D\boldsymbol{u}_\varepsilon}|} \right)\left(\frac{{D\boldsymbol{u}_\varepsilon}}{|{D\boldsymbol{u}_\varepsilon}|}:{D\boldsymbol \varphi} \right)\,dx
        \\&+\int_{B_{R}}(|{D\boldsymbol{u}_\varepsilon}|-1)_+^{p-1}\left(\frac{{D\boldsymbol{u}_{\varepsilon,i}}}{|{D\boldsymbol{u}_\varepsilon}|}:{D\boldsymbol \varphi} \right)\,dx
        \\&-\int_{B_{R}}(|{D\boldsymbol{u}_\varepsilon}|-1)_+^{p-1}\left(\frac{{D\boldsymbol{u}_{\varepsilon,i}}}{|{D\boldsymbol{u}_\varepsilon}|}:\frac{{D\boldsymbol{u}_\varepsilon}}{|{D\boldsymbol{u}_\varepsilon}|} \right)\left(\frac{{D\boldsymbol{u}_\varepsilon}}{|{D\boldsymbol{u}_\varepsilon}|}:{D\boldsymbol \varphi} \right)\,dx.
\end{split}
    \end{equation}
  
 Now we estimate the right-hand side of \eqref{cose2}. To this end, we recall the expression of the gradient of the test function $\varphi$ defined in \eqref{testphi}, namely \eqref{gradientetestphi}. In particular, the terms arising from \({\mathcal{J}}_1\) and \({\mathcal{J}}_2\) will be estimated from below, whereas the terms coming from \({\mathcal{J}}_3\) will be estimated from above. Considering the terms in \eqref{cose2} involving ${\mathcal{J}}_1$, we have

\begin{equation*}\label{primipezzi}
    \begin{split}
        &\int_{B_{R}}(p-1)\left(1-\alpha\frac{(|D\boldsymbol u_\varepsilon|-1-\delta)_+}{(|D\boldsymbol u_\varepsilon|-1)_+}\right)(|{D\boldsymbol{u}_\varepsilon}|-1)_+^{p-2-\alpha}\\&\qquad\qquad\qquad\times\left({D\boldsymbol{u}_{\varepsilon,i}}:\frac{{D\boldsymbol{u}_\varepsilon}}{|{D\boldsymbol{u}_\varepsilon}|} \right)\left(\frac{{D\boldsymbol{u}_\varepsilon}}{|{D\boldsymbol{u}_\varepsilon}|}: \frac{{{\boldsymbol u}_{\varepsilon,i}}\otimes{D^2\boldsymbol{u}_\varepsilon}{D\boldsymbol{u}_\varepsilon}}{|{D\boldsymbol{u}_\varepsilon}|} \right){\psi}^2\chi_{\delta,\varepsilon}\,dx
        \\&+\int_{B_{R}}\left(1-\alpha\frac{(|D\boldsymbol u_\varepsilon|-1-\delta)_+}{(|D\boldsymbol u_\varepsilon|-1)_+}\right)(|{D\boldsymbol{u}_\varepsilon}|-1)_+^{p-1-\alpha}\\&\qquad\qquad\qquad\times\left(\frac{{D\boldsymbol{u}_{\varepsilon,i}} }{|{D\boldsymbol{u}_\varepsilon}|}: \frac{{\boldsymbol u}_{\varepsilon,i}\otimes{D^2\boldsymbol{u}_\varepsilon}{D\boldsymbol{u}_\varepsilon}}{|{D\boldsymbol{u}_\varepsilon}|} \right){\psi}^2\chi_{\delta,\varepsilon}\,dx
        \\&-\int_{B_{R}}\left(1-\alpha\frac{(|D\boldsymbol u_\varepsilon|-1-\delta)_+}{(|D\boldsymbol u_\varepsilon|-1)_+}\right)(|{D\boldsymbol{u}_\varepsilon}|-1)_+^{p-1-\alpha}\\&\qquad\qquad\qquad\times\left(\frac{{D\boldsymbol{u}_{\varepsilon,i}} }{|{D\boldsymbol{u}_\varepsilon}|}:\frac{{D\boldsymbol{u}_\varepsilon}}{|{D\boldsymbol{u}_\varepsilon}|} \right)\left(\frac{{D\boldsymbol{u}_\varepsilon}}{|{D\boldsymbol{u}_\varepsilon}|}:\frac{{\boldsymbol u}_{\varepsilon,i}\otimes{D^2\boldsymbol{u}_\varepsilon}{D\boldsymbol{u}_\varepsilon}}{|{D\boldsymbol{u}_\varepsilon}|} \right){\psi}^2\chi_{\delta,\varepsilon}\,dx.
\end{split}
\end{equation*}
Recalling \eqref{eq:vect} and denoting by \((D^{2}\boldsymbol{u}_\varepsilon\, D\boldsymbol{u}_\varepsilon)_t\) the \(t\)-th component of the corresponding vector, we note that 
\begin{equation}\label{stima2}
    \begin{split}
&\sum_{i=1}^n\left({D\boldsymbol{u}_{\varepsilon,i}}:D\boldsymbol{u}_\varepsilon \right)\left(D\boldsymbol{u}_\varepsilon: {{\boldsymbol u}_{\varepsilon,i}}\otimes{D^2\boldsymbol{u}_\varepsilon}{D\boldsymbol{u}_\varepsilon} \right)
\\&=\sum_{i,k,t=1}^n\sum_{l,s=1}^Nu_{\varepsilon,ik}^lu_{\varepsilon,k}^lu_{\varepsilon,t}^su_{\varepsilon,i}^s(D^2\boldsymbol u_\varepsilon D\boldsymbol u_\varepsilon)_t
\\&=\sum_{i,k,t=1}^n\sum_{l,s=1}^Nu_{\varepsilon,ik}^lu_{\varepsilon,i}^su_{\varepsilon,k}^lu_{\varepsilon,t}^s(D\boldsymbol u_\varepsilon: D\boldsymbol u_{\varepsilon,t})
\\&=\sum_{i,t=1}^n\sum_{s=1}^Nu_{\varepsilon,i}^s(D\boldsymbol u_\varepsilon: D\boldsymbol u_{\varepsilon,i})u_{\varepsilon,t}^s(D\boldsymbol u_\varepsilon: D\boldsymbol u_{\varepsilon,t})
\\&=\sum_{s=1}^N((D\boldsymbol u_\varepsilon (D^2\boldsymbol u_\varepsilon D\boldsymbol u_\varepsilon))_s)^2=|D\boldsymbol u_\varepsilon (D^2\boldsymbol u_\varepsilon D\boldsymbol u_\varepsilon)|^2=|\Delta_{\infty}\boldsymbol{u}|^2.
    \end{split}
\end{equation}
By \eqref{stima1} and \eqref{stima2}, taking the sum for $i=1,...,n$, we get 
\begin{equation*}
    \begin{split}
        &\int_{B_{R}}(p-1)\left(1-\alpha\frac{(|D\boldsymbol u_\varepsilon|-1-\delta)_+}{(|D\boldsymbol u_\varepsilon|-1)_+}\right)(|{D\boldsymbol{u}_\varepsilon}|-1)_+^{p-2-\alpha}\left| \frac{\Delta_{\infty}\boldsymbol{u}_\varepsilon}{|{D\boldsymbol{u}_\varepsilon}|^2} \right|^2|{D\boldsymbol{u}_\varepsilon}|{\psi}^2\chi_{\delta,\varepsilon}\,dx
        \\&+\int_{B_{R}}\left(1-\alpha\frac{(|D\boldsymbol u_\varepsilon|-1-\delta)_+}{(|D\boldsymbol u_\varepsilon|-1)_+}\right)(|{D\boldsymbol{u}_\varepsilon}|-1)_+^{p-1-\alpha}\left| \frac{{D^2\boldsymbol{u}_\varepsilon}{D\boldsymbol{u}_\varepsilon}}{|{D\boldsymbol{u}_\varepsilon}|} \right|^2{\psi}^2\chi_{\delta,\varepsilon}\,dx
        \\&-\int_{B_{R}}\left(1-\alpha\frac{(|D\boldsymbol u_\varepsilon|-1-\delta)_+}{(|D\boldsymbol u_\varepsilon|-1)_+}\right)(|{D\boldsymbol{u}_\varepsilon}|-1)_+^{p-1-\alpha}\left| \frac{\Delta_{\infty}\boldsymbol{u}_\varepsilon}{|{D\boldsymbol{u}_\varepsilon}|^2} \right|^2{\psi}^2\chi_{\delta,\varepsilon}\,dx.
\end{split}
\end{equation*}
Since $\alpha>0$ and $(|D\boldsymbol u_\varepsilon|-1-\delta)_+\leq (|D\boldsymbol u_\varepsilon|-1)_+$, the latter terms can be estimated from below as
\begin{equation*}
    \begin{split}
        &\left(1-\alpha\right)\int_{B_{R}}(p-1)(|{D\boldsymbol{u}_\varepsilon}|-1)_+^{p-2-\alpha}\left| \frac{\Delta_{\infty}\boldsymbol{u}_\varepsilon}{|{D\boldsymbol{u}_\varepsilon}|^2} \right|^2|{D\boldsymbol{u}_\varepsilon}|{\psi}^2\chi_{\delta,\varepsilon}\,dx
        \\&+\left(1-\alpha\right)\int_{B_{R}}(|{D\boldsymbol{u}_\varepsilon}|-1)_+^{p-1-\alpha}\left| \frac{{D^2\boldsymbol{u}_\varepsilon}{D\boldsymbol{u}_\varepsilon}}{|{D\boldsymbol{u}_\varepsilon}|} \right|^2{\psi}^2\chi_{\delta,\varepsilon}\,dx
        \\&-\left(1-\alpha\right)\int_{B_{R}}(|{D\boldsymbol{u}_\varepsilon}|-1)_+^{p-1-\alpha}\left| \frac{\Delta_{\infty}\boldsymbol{u}_\varepsilon}{|{D\boldsymbol{u}_\varepsilon}|^2} \right|^2{\psi}^2\chi_{\delta,\varepsilon}\,dx\\&:={\mathcal{J}}_{1,1}+{\mathcal{J}}_{1,2}+{\mathcal{J}}_{1,3}.
    \end{split}
\end{equation*}
Since $\alpha<1$, the integral $\mathcal{J}_{1,2}$ is non–negative. Using the Cauchy-Schwarz's inequality, we deduce that ${\mathcal{J}}_{1,3}\leq {\mathcal{J}}_{1,2}$. Therefore, we get the following inequality
\begin{equation}\label{pezzoimportantino}
    {\mathcal{J}}_{1,1}+{\mathcal{J}}_{1,2}+{\mathcal{J}}_{1,3}\geq {\mathcal{J}}_{1,1}.
\end{equation}

Considering the terms involving $\mathcal{J}_2$ in the right-hand side of \eqref{cose2}, we have 
\begin{equation*}\label{secondipezzi}
    \begin{split}
        &\int_{B_{R}}(p-1)(|{D\boldsymbol{u}_\varepsilon}|-1)_+^{p-2-\alpha}(|{D\boldsymbol{u}_\varepsilon}|-1-\delta)_+\left({D\boldsymbol{u}_{\varepsilon,i}}:\frac{{D\boldsymbol{u}_\varepsilon}}{|{D\boldsymbol{u}_\varepsilon}|} \right)\left(\frac{{D\boldsymbol{u}_\varepsilon}}{|{D\boldsymbol{u}_\varepsilon}|}:{{D\boldsymbol{u}_{\varepsilon,i}}}  \right){\psi}^2\,dx\\&+\int_{B_{R}}(|{D\boldsymbol{u}_\varepsilon}|-1)_+^{p-1-\alpha}(|{D\boldsymbol{u}_\varepsilon}|-1-\delta)_+\left(\frac{{D\boldsymbol{u}_{\varepsilon,i}}}{|{D\boldsymbol{u}_\varepsilon}|}: {{D\boldsymbol{u}_{\varepsilon,i}}}\right){\psi}^2\,dx
        \\&-\int_{B_{R}}(|{D\boldsymbol{u}_\varepsilon}|-1)_+^{p-1-\alpha}(|{D\boldsymbol{u}_\varepsilon}|-1-\delta)_+\left(\frac{{D\boldsymbol{u}_{\varepsilon,i}}}{|{D\boldsymbol{u}_\varepsilon}|}:\frac{{D\boldsymbol{u}_\varepsilon}}{|{D\boldsymbol{u}_\varepsilon}|} \right)\left(\frac{{D\boldsymbol{u}_\varepsilon}}{|{D\boldsymbol{u}_\varepsilon}|}:  {{D\boldsymbol{u}_{\varepsilon,i}}}\right){\psi}^2\,dx
        \end{split}
\end{equation*}
By recalling \eqref{eq:vect} and summing over the indices $i=1,\ldots,n$, we obtain
\begin{equation*}\label{secondipezzi2}
    \begin{split}
        &\int_{B_{R}}(p-1)(|{D\boldsymbol{u}_\varepsilon}|-1)_+^{p-2-\alpha}(|{D\boldsymbol{u}_\varepsilon}|-1-\delta)_+\left|\frac{{D^2\boldsymbol u_\varepsilon}{D\boldsymbol{u}_\varepsilon}}{|{D\boldsymbol{u}_\varepsilon}|} \right|^2{\psi}^2\,dx\\&+\int_{B_{R}}(|{D\boldsymbol{u}_\varepsilon}|-1)_+^{p-1-\alpha}(|{D\boldsymbol{u}_\varepsilon}|-1-\delta)_+\|{D^2\boldsymbol u_\varepsilon}\|^2\frac{1}{|{D\boldsymbol{u}_\varepsilon}|}{\psi}^2\,dx\\&-\int_{B_{R}}(|{D\boldsymbol{u}_\varepsilon}|-1)_+^{p-1-\alpha}(|{D\boldsymbol{u}_\varepsilon}|-1-\delta)_+\left|\frac{{D^2\boldsymbol u_\varepsilon}{D\boldsymbol{u}_\varepsilon}}{|{D\boldsymbol{u}_\varepsilon}|}\right|^2\frac{1}{|{D\boldsymbol{u}_\varepsilon}|}{\psi}^2\,dx\\&={\mathcal{J}}_{2,1}+{\mathcal{J}}_{2,2}+{\mathcal{J}}_{2,3}.
\end{split}
\end{equation*}
We note that, by using Cauchy-Schwarz's inequality we get ${\mathcal{J}}_{2,3}\leq {\mathcal{J}}_{2,2}.$
Therefore, we deduce
\begin{equation}\label{pezzotildeJ_2}
    {\mathcal{J}}_{2,1}+{\mathcal{J}}_{2,2}+{\mathcal{J}}_{2,3}\geq {\mathcal{J}}_{2,1}.
\end{equation}

We now analyze the terms coming from ${\mathcal{J}}_3$ in the right-hand side of \eqref{cose2}. Taking the sum over \(i=1,\dots,n\) and applying Cauchy–Schwarz's inequality, we get

\begin{equation*}\label{terzipezzi}
    \begin{split}
&2\int_{B_{R}}(p-1)(|{D\boldsymbol{u}_\varepsilon}|-1)_+^{p-2-\alpha}(|{D\boldsymbol{u}_\varepsilon}|-1-\delta)_+\left({D\boldsymbol{u}_{\varepsilon,i}}:\frac{{D\boldsymbol{u}_\varepsilon}}{|{D\boldsymbol{u}_\varepsilon}|} \right)\left(\frac{{D\boldsymbol{u}_\varepsilon}}{|{D\boldsymbol{u}_\varepsilon}|}:{{\boldsymbol u}_{\varepsilon,i}}\otimes\nabla\psi \right){\psi}\,dx
\\&+2\int_{B_{R}}(|{D\boldsymbol{u}_\varepsilon}|-1)_+^{p-1-\alpha}(|{D\boldsymbol{u}_\varepsilon}|-1-\delta)_+\left(\frac{{D\boldsymbol{u}_{\varepsilon,i}}}{|{D\boldsymbol{u}_\varepsilon}|}: {{\boldsymbol u}_{\varepsilon,i}}\otimes\nabla\psi\right){\psi}\,dx
\\&-2\int_{B_{R}}(|{D\boldsymbol{u}_\varepsilon}|-1)_+^{p-1-\alpha}(|{D\boldsymbol{u}_\varepsilon}|-1-\delta)_+\left(\frac{{D\boldsymbol{u}_{\varepsilon,i}}}{|{D\boldsymbol{u}_\varepsilon}|}:\frac{{D\boldsymbol{u}_\varepsilon}}{|{D\boldsymbol{u}_\varepsilon}|} \right)\left(\frac{{D\boldsymbol{u}_\varepsilon}}{|{D\boldsymbol{u}_\varepsilon}|}: {{\boldsymbol u}_{\varepsilon,i}}\otimes\nabla\psi\right){\psi}\,dx
\\&\leq  C(p)\int_{B_{R}}(|{D\boldsymbol{u}_\varepsilon}|-1)_+^{p-2-\alpha}(|{D\boldsymbol{u}_\varepsilon}|-1-\delta)_+\left|{D^2\boldsymbol u_\varepsilon}\frac{{D\boldsymbol{u}_\varepsilon}}{|{D\boldsymbol{u}_\varepsilon}|} \right||{D\boldsymbol{u}_\varepsilon}||\nabla\psi| {\psi}\,dx
\\&+C(p)\int_{B_{R}}(|{D\boldsymbol{u}_\varepsilon}|-1)_+^{p-1-\alpha}(|{D\boldsymbol{u}_\varepsilon}|-1-\delta)_+\left|{D^2\boldsymbol u_\varepsilon}\frac{{D\boldsymbol{u}_\varepsilon}}{|{D\boldsymbol{u}_\varepsilon}|} \right||\nabla\psi|{\psi}\,dx=:{\mathcal{J}}_{3,1}+{\mathcal{J}}_{3,2}.
        \end{split}
\end{equation*}

Since $|D\boldsymbol{u}_\varepsilon|\geq (|{D\boldsymbol{u}_\varepsilon}|-1)_+$ we deduce that \({\mathcal{J}}_{3,2} \leq {\mathcal{J}}_{3,1}\). So, the preceding terms can be estimated by \({\mathcal{J}}_{3,1}\). Using the weighted Young's inequality, we deduce 
\begin{equation}\label{pezzo3}
    \begin{split}
{\mathcal{J}}_{3,1}&\leq \hat{\varepsilon}\int_{B_{R}}(|{D\boldsymbol{u}_\varepsilon}|-1)_+^{p-2-\alpha}(|{D\boldsymbol{u}_\varepsilon}|-1-\delta)_+\left|\frac{{D^2\boldsymbol u_\varepsilon}{D\boldsymbol{u}_\varepsilon}}{|{D\boldsymbol{u}_\varepsilon}|} \right|^2{\psi}^2\,dx
\\&+ C(p,\hat{\varepsilon})\int_{B_{R}}(|{D\boldsymbol{u}_\varepsilon}|-1)_+^{p-2-\alpha}(|{D\boldsymbol{u}_\varepsilon}|-1-\delta)_+|{D\boldsymbol{u}_\varepsilon}|^2|\nabla\psi|^2\,dx
        \end{split}
\end{equation}

Finally, by combining \eqref{pezzoimportantino}, \eqref{pezzotildeJ_2} and \eqref{pezzo3},  and plugging them into \eqref{cose2}, we obtain

\begin{equation*}
    \begin{split}
    &(1-\alpha)\int_{B_{R}}(p-1)(|{D\boldsymbol{u}_\varepsilon}|-1)_+^{p-2-\alpha}\left| \frac{\Delta_{\infty}\boldsymbol{u}_\varepsilon}{|{D\boldsymbol{u}_\varepsilon}|^2} \right|^2|{D\boldsymbol{u}_\varepsilon}|{\psi}^2\chi_{\delta,\varepsilon}\,dx
    \\&+(p-1-\hat{\varepsilon})\int_{B_{R}}(|{D\boldsymbol{u}_\varepsilon}|-1)_+^{p-2-\alpha}(|{D\boldsymbol{u}_\varepsilon}|-1-\delta)_+\left|\frac{{D^2\boldsymbol u_\varepsilon}{D\boldsymbol{u}_\varepsilon}}{|{D\boldsymbol{u}_\varepsilon}|} \right|^2{\psi}^2\,dx
    \\&\leq C(p,\hat{\varepsilon})\int_{B_{R}}(|{D\boldsymbol{u}_\varepsilon}|-1)_+^{p-2-\alpha}(|{D\boldsymbol{u}_\varepsilon}|-1-\delta)_+|{D\boldsymbol{u}_\varepsilon}|^2|\nabla\psi|^2\,dx
\\&+\sum_{i=1}^n\int_{B_{R}}\left(\partial_{x_i}\left((|D\boldsymbol{u}_\varepsilon|-1)_+^{p-1}\frac{D\boldsymbol{u}_\varepsilon}{|D\boldsymbol{u}_\varepsilon|}\right):D\boldsymbol{\varphi} \right)\,dx
    \end{split}
\end{equation*}
Note that, since $u_\varepsilon\in W^{1,\infty}_{loc}(B_R)$, the first term on the r.h.s. of the previous inequality is bounded for $\alpha<\min\{1,p-1\}$. Using the fact that \((|D\boldsymbol{u}_\varepsilon|-1-\delta)_+ \leq (|D\boldsymbol{u}_\varepsilon|-1)_+\) in the right-hand side of the previous inequality, and then letting \(\delta \to 0\), we obtain the thesis.
\end{proof}

The preceding two lemmas allow us to derive the following proposition.
\begin{proposition}\label{bellazio}
   Let \(\Omega\) be an open set in \(\mathbb{R}^n\) and assume that \(B_{R}\subset\subset \Omega\). Let \(\boldsymbol{u}_\varepsilon\) be a weak solution of \eqref{eq:problregol} and let \(\boldsymbol{\varphi}\) be the function defined in \eqref{testphi}. Moreover, we assume that $\boldsymbol{f}\in W^{1,1}_{loc}(\Omega)\cap L^q_{loc}(\Omega)$, with $q>n$. Then, for $0<\alpha<\min \{1,p-1\}$, the following inequality holds
\begin{equation}\label{bombolone}
    \begin{split}
    &\int_{B_{R}} \varepsilon{(|D\boldsymbol u_\varepsilon|-1)_+^{-\alpha}}\left|\frac{{D^2\boldsymbol{u}_\varepsilon}{D\boldsymbol{u}_\varepsilon}}{|{D\boldsymbol{u}_\varepsilon}|}\right|^2|{D\boldsymbol{u}_\varepsilon}|{\psi}^2\chi_{0,\varepsilon}\,dx
      \\&+\int_{B_{R}}(|{D\boldsymbol{u}_\varepsilon}|-1)_+^{p-2-\alpha}\left| \frac{\Delta_{\infty}\boldsymbol{u}_\varepsilon}{|{D\boldsymbol{u}_\varepsilon}|^2} \right|^2|{D\boldsymbol{u}_\varepsilon}|{\psi}^2\chi_{0,\varepsilon}\,dx
     \\&\leq C(\alpha)\int_{B_{R}} \varepsilon{(|{D\boldsymbol{u}_\varepsilon}|-1)_+^{1-\alpha}}|{D\boldsymbol{u}_\varepsilon}|^2|\nabla\psi|^2\,dx
+\\&+C(p,\alpha)\int_{B_{R}}(|{D\boldsymbol{u}_\varepsilon}|-1)_+^{p-1-\alpha}|{D\boldsymbol{u}_\varepsilon}|^2|\nabla\psi|^2\,dx+C(p,\alpha)\sum_{i=1}^n\int_{B_{R}}\langle \boldsymbol{f}_{i},\boldsymbol{\varphi}\rangle\,dx,
    \end{split}
\end{equation}
where $C(\alpha)$ and $C(p,\alpha)$ are positive constants.
\end{proposition}

 \begin{proof}
 We recall that if $\boldsymbol{f}\in L^{q}(B_{R})$ with $q>n$, then by \cite[Lemma~3.3]{BDGP1} the function $\boldsymbol{\varphi}$ defined in \eqref{testphi} is an admissible test function. We start the proof noting that the equation \eqref{primoregolarizzato} can be rewritten as
\begin{equation}\label{a_4}
    \begin{split}
    &\int_{B_{R}} \varepsilon\left( {D\boldsymbol{u}_{\varepsilon,i}}:{D\boldsymbol \varphi}\right)\,dx+\int_{B_{R}}\partial_{x_i}\left((|{D\boldsymbol{u}_\varepsilon}|-1)_+^{p-1}\frac{{D\boldsymbol{u}_\varepsilon}}{|{D\boldsymbol{u}_\varepsilon}|} \right):{D\boldsymbol \varphi}\,dx\\&=\int_{B_{R}}\langle \boldsymbol{f}_{i},\boldsymbol{\varphi}\rangle\,dx.
\end{split}
    \end{equation}

By Lemma \ref{Lemma11} and Lemma \ref{Lemma2}, from the equation \eqref{a_4} we deduce that
\begin{equation}\label{zio}
    \begin{split}
      &\int_{B_{R}} \varepsilon{(|D\boldsymbol u_\varepsilon|-1)_+^{-\alpha}}\left|\frac{{D^2\boldsymbol{u}_\varepsilon}{D\boldsymbol{u}_\varepsilon}}{|{D\boldsymbol{u}_\varepsilon}|}\right|^2|{D\boldsymbol{u}_\varepsilon}|{\psi}^2\chi_{0,\varepsilon}\,dx
          +\int_{B_{R}} \varepsilon{(|D\boldsymbol u_\varepsilon|-1)_+^{1-\alpha}}\|{D^2\boldsymbol{u}_\varepsilon}\|^2{\psi}^2\,dx
          \\&+\int_{B_{R}}(|{D\boldsymbol{u}_\varepsilon}|-1)_+^{p-2-\alpha}\left| \frac{\Delta_{\infty}\boldsymbol{u}_\varepsilon}{|{D\boldsymbol{u}_\varepsilon}|^2} \right|^2|{D\boldsymbol{u}_\varepsilon}|{\psi}^2\chi_{0,\varepsilon}\,dx
    +\int_{B_{R}}(|{D\boldsymbol{u}_\varepsilon}|-1)_+^{p-1-\alpha}\left|\frac{{D^2\boldsymbol u_\varepsilon}{D\boldsymbol{u}_\varepsilon}}{|{D\boldsymbol{u}_\varepsilon}|} \right|^2{\psi}^2\,dx
         \\&\leq C(\alpha)\int_{B_{R}} \varepsilon{{(|D\boldsymbol u_\varepsilon|-1)_+^{1-\alpha}}}|{D\boldsymbol{u}_\varepsilon}|^2|\nabla\psi|^2\,dx
\\&+C(p,\alpha)\int_{B_{R}}(|{D\boldsymbol{u}_\varepsilon}|-1)_+^{p-1-\alpha}|{D\boldsymbol{u}_\varepsilon}|^2|\nabla\psi|^2\,dx+C(p,\alpha)\sum_{i=1}^n\int_{B_{R}}\langle \boldsymbol{f}_{i},\boldsymbol{\varphi}\rangle\,dx.
\end{split}
\end{equation}
Since the second and the fourth integrals in the l.h.s. of the previous inequality are positive the latter can be rewritten as
\begin{equation*}
    \begin{split}
    &\int_{B_{R}} \varepsilon{(|D\boldsymbol u_\varepsilon|-1)_+^{-\alpha}}\left|\frac{{D^2\boldsymbol{u}_\varepsilon}{D\boldsymbol{u}_\varepsilon}}{|{D\boldsymbol{u}_\varepsilon}|}\right|^2|{D\boldsymbol{u}_\varepsilon}|{\psi}^2\chi_{0,\varepsilon}\,dx
      \\&+\int_{B_{R}}(|{D\boldsymbol{u}_\varepsilon}|-1)_+^{p-2-\alpha}\left| \frac{\Delta_{\infty}\boldsymbol{u}_\varepsilon}{|{D\boldsymbol{u}_\varepsilon}|^2} \right|^2|{D\boldsymbol{u}_\varepsilon}|{\psi}^2\chi_{0,\varepsilon}\,dx
     \\&\leq C(\alpha)\int_{B_{R}} \varepsilon{(|{D\boldsymbol{u}_\varepsilon}|-1)_+^{1-\alpha}}|{D\boldsymbol{u}_\varepsilon}|^2|\nabla\psi|^2\,dx
+\\&+C(p,\alpha)\int_{B_{R}}(|{D\boldsymbol{u}_\varepsilon}|-1)_+^{p-1-\alpha}|{D\boldsymbol{u}_\varepsilon}|^2|\nabla\psi|^2\,dx+C(p,\alpha)\sum_{i=1}^n\int_{B_{R}}\langle \boldsymbol{f}_{i},\boldsymbol{\varphi}\rangle\,dx.
\end{split}
\end{equation*}

\end{proof}

We are now in a position to proceed with the core of the Section.

\begin{proof}[Proof of Theorem \ref{TeoDerivateSeconde}]
Let $\tilde{\Omega} \subset\subset \Omega$ and choose $R>0$ such that $\tilde{\Omega} \subset\subset B_R \subset\subset \Omega$. Let $\boldsymbol{u}_\varepsilon$ be a weak solution to the regularized problem \eqref{eq:problregol} in $B_R$. By \cite[Lemma~3.1 and Lemma~3.2]{BDGP1}, we have the uniform bound
\begin{equation}\label{stimadelgradiente}
\|D\boldsymbol{u}_\varepsilon\|_{L^{\infty}(\tilde \Omega)}\leq C(n,N,p,\tilde R,\|\boldsymbol{f}\|_{L^q},\|D\boldsymbol{u}\|_{L^p}).
\end{equation}

Let $\delta>0$ be fixed. We define the non-degenerate open set
\begin{equation*}
\Omega_\delta :=  \{ x \in \tilde{\Omega} : |D\boldsymbol{u}(x)| > 1+2\delta \}.
\end{equation*}
By \cite[Theorem~1.1]{BDGP1}, the function $\mathcal K(D\boldsymbol{u}) = (|D\boldsymbol{u}|-1)_+$ is continuous, hence $\Omega_\delta$ is an open set strictly separated from the singular region. Now we set 
\begin{equation*}
\begin{split}
 \boldsymbol{v}_\varepsilon:=(|D\boldsymbol{u}_\varepsilon|-1-\delta)_+\frac{D\boldsymbol{u}_\varepsilon}{|D\boldsymbol{u}_\varepsilon|},\quad \boldsymbol{v}:=(|D\boldsymbol{u}|-1-\delta)_+\frac{D\boldsymbol{u}}{|D\boldsymbol{u}|}.
 \end{split}
 \end{equation*}
By \cite[Lemma~3.3]{BDGP1} (see in particular the beginning of the proof of Theorem~1.1 therein), up to subsequences, we have \(\boldsymbol v_\varepsilon \rightarrow \boldsymbol v\) in the norm $C^{0,\alpha(\delta)}(\overline{\Omega}_\delta)$. In the following we will repeatedly use the following interpolation inequality: 
if $g\in C^{0,\alpha}$ and $0<\alpha'<\alpha$, then
\begin{equation}\label{disinterpolazione}
[g]_{C^{0,\alpha'}} \leq C\, \|g\|_{L^\infty}^{\,1-\frac{\alpha'}{\alpha}}\, [g]_{C^{0,\alpha}}^{\,\frac{\alpha'}{\alpha}},
\end{equation}
where $C>0$ is a constant and $[\cdot]$ denotes the Hölder seminorm. By \eqref{disinterpolazione} and by the reverse triangle inequality, we deduce that $|\boldsymbol v_\varepsilon| \rightarrow |\boldsymbol v|$ in the norm $C^{0,\alpha'(\delta)}(\overline{\Omega}_\delta)$, with $0<\alpha'(\delta)<\alpha(\delta)$.
Moreover, we note that, there exists $\bar\varepsilon>0$ small enough, such that $|D\boldsymbol u_\varepsilon|\geq 1+3\delta/2$ in $\overline{\Omega}_\delta$, for every $\varepsilon<\bar\varepsilon$. Therefore, we deduce $|D\boldsymbol u_\varepsilon| \rightarrow |D\boldsymbol u|$ in the norm $C^{0,\alpha'(\delta)}(\overline{\Omega}_\delta)$. By the previous considerations and by \eqref{disinterpolazione}, it follows that
\[
\frac{D\boldsymbol u_\varepsilon}{|D\boldsymbol u_\varepsilon|}=\frac{\boldsymbol v_\varepsilon}{|\boldsymbol v_\varepsilon|}\rightarrow\frac{\boldsymbol v}{|\boldsymbol v|}=\frac{D\boldsymbol u}{|D\boldsymbol u|} \qquad\text{in } C^{0,\alpha{''}(\delta)}(\overline{\Omega}_\delta),
\]
with $0<\alpha{''}(\delta)<\alpha{'}(\delta)$. Combining the convergence of the modulus and of the direction, by \eqref{disinterpolazione}, we finally obtain
\begin{equation}\label{convergenzadeigradienti}
D\boldsymbol u_\varepsilon=|D\boldsymbol u_\varepsilon|\,\frac{D\boldsymbol u_\varepsilon}{|D\boldsymbol u_\varepsilon|}\rightarrow|D\boldsymbol u|\,\frac{D\boldsymbol u}{|D\boldsymbol u|}=D\boldsymbol u \qquad\text{in } C^{0,\alpha'''(\delta)}(\overline{\Omega}_\delta),
\end{equation}
with $0<\alpha{'''}(\delta)<\alpha{''}(\delta)$. Because the regularized system is strictly elliptic on $\Omega_\delta$ with uniformly Hölder continuous coefficients, standard elliptic regularity (see \cite[Theorem~9.11]{GT} applied locally to each scalar equation of the system) yields
$\| \boldsymbol u_\varepsilon \|_{W^{2,2}(\Omega_\delta)} \leq C$, where $C>0$ is a constant independent of $\varepsilon$. Therefore, up to a subsequence,
\begin{equation}\label{convergenzaW22}
    \boldsymbol u_\varepsilon \rightharpoonup \boldsymbol u \qquad\text{weakly in } W^{2,2}(\Omega_\delta).
\end{equation}

On the other hand, from \eqref{stimadelgradiente} and Proposition~\ref{bellazio}, choosing a cut-off function $\psi$ such that $\psi \equiv 1$ on $\tilde{\Omega}$, we have the global uniform bound:
\begin{equation}\label{stimauniforme}
\int_{\tilde \Omega}(|{D\boldsymbol{u}_\varepsilon}|-1)_+^{p-2-\alpha}\left| \frac{\Delta_{\infty}\boldsymbol{u}_\varepsilon}{|{D\boldsymbol{u}_\varepsilon}|^2} \right|^2|{D\boldsymbol{u}_\varepsilon}|\chi_{0,\varepsilon}\,dx\leq C,
\end{equation}
where $C$ does not depend on $\varepsilon$ nor on $\delta$. Since the integrand is non-negative, we can restrict the domain of integration to $\Omega_\delta$. We rewrite the restricted integral as
\begin{equation*}
\left\|T_{D\boldsymbol{u}_\varepsilon}(D^2\boldsymbol{u}_\varepsilon)\right\|_{L^2(\Omega_\delta)}^2 \leq C,
\end{equation*}
where we define the linear operator $T_K : \mathbb{R}^{Nn^2} \to \mathbb{R}^N$ as
\[T_{K} M:=(|{K}|-1)_+^{\frac{p-2-\alpha}{2}} |{K}|^{\frac 12} \frac{{K}}{|K|} \left( \frac{M K}{|K|} \right),\quad K\in\{D\boldsymbol u_\varepsilon\,,\,D\boldsymbol u\}.\]
Since $D\boldsymbol{u}_\varepsilon$ is uniformly bounded and strictly away from $1$ on $\Omega_\delta$, the linear operator $T_{D\boldsymbol{u}_\varepsilon}$ is uniformly bounded on $L^2(\Omega_\delta)$. As established before, we have $T_{D\boldsymbol{u}}(D^2\boldsymbol u_\varepsilon) \rightharpoonup T_{D\boldsymbol{u}}(D^2\boldsymbol u)$ weakly in $L^{2}(\Omega_\delta)$. Hence, by the weak lower semi-continuity of the norm squared, we have
\begin{equation}\label{bombolo}
\begin{split}
&\left\|T_{D\boldsymbol{u}}(D^2\boldsymbol{u})\right\|_{L^2(\Omega_\delta)}^2 \leq \liminf_{\varepsilon\rightarrow 0}\left\|T_{D\boldsymbol{u}}(D^2\boldsymbol{u}_\varepsilon)\right\|_{L^2(\Omega_\delta)}^2
\\&=\liminf_{\varepsilon\rightarrow 0}\left( \left\|T_{D\boldsymbol{u}_\varepsilon}(D^2\boldsymbol{u}_\varepsilon)\right\|_{L^2(\Omega_\delta)}^2 + \left\|T_{D\boldsymbol{u}}(D^2\boldsymbol{u}_\varepsilon)\right\|_{L^2(\Omega_\delta)}^2 - \left\|T_{D\boldsymbol{u}_\varepsilon}(D^2\boldsymbol{u}_\varepsilon)\right\|_{L^2(\Omega_\delta)}^2 \right).
\end{split}
\end{equation}
By \eqref{stimauniforme}, the first term inside the limit inferior is bounded by $C$. We just need to prove that the difference vanishes, namely:
\begin{equation*}
\lim_{\varepsilon\rightarrow 0} \left| \left\|T_{D\boldsymbol{u}}(D^2\boldsymbol{u}_\varepsilon)\right\|_{L^2(\Omega_\delta)}^2 - \left\|T_{D\boldsymbol{u}_\varepsilon}(D^2\boldsymbol{u}_\varepsilon)\right\|_{L^2(\Omega_\delta)}^2 \right| =0.
\end{equation*}
Indeed, we can rewrite the argument of the limit as
\begin{equation*}
\begin{split}
&\left|\langle T_{D\boldsymbol{u}}(D^2\boldsymbol{u}_\varepsilon), T_{D\boldsymbol{u}}(D^2\boldsymbol{u}_\varepsilon) \rangle_{L^2(\Omega_\delta)}-\langle T_{D\boldsymbol{u}_\varepsilon}(D^2\boldsymbol{u}_\varepsilon),T_{D\boldsymbol{u}_\varepsilon}(D^2\boldsymbol{u}_\varepsilon) \rangle_{L^2(\Omega_\delta)}\right|
\\&=\left|\langle (T^*_{D\boldsymbol{u}}T_{D\boldsymbol{u}}-T^*_{D\boldsymbol{u}_\varepsilon}T_{D\boldsymbol{u}_\varepsilon})D^2\boldsymbol{u}_\varepsilon, D^2\boldsymbol{u}_\varepsilon \rangle_{L^2(\Omega_\delta)}\right|
\\&\leq \left\|T^*_{D\boldsymbol{u}}T_{D\boldsymbol{u}}-T^*_{D\boldsymbol{u}_\varepsilon}T_{D\boldsymbol{u}_\varepsilon}\right\|_{L^{\infty} (\Omega_\delta)}\|D^2\boldsymbol{u}_\varepsilon\|_{L^{2} (\Omega_\delta)}^2.
\end{split}
\end{equation*}
By \eqref{convergenzadeigradienti} and the uniform $W^{2,2}$ bound \eqref{convergenzaW22}, the above quantity goes to zero as $\varepsilon\rightarrow 0$. Thus, since $|D\boldsymbol{u}| \geq 1$ on $\Omega_\delta$, we deduce that
\begin{equation}\label{stimaOmegadelta}
\int_{\Omega_\delta}(|{D\boldsymbol{u}}|-1)_+^{p-2-\alpha}\left|\Delta_{\infty}\boldsymbol u\right|^2\,dx \leq \int_{\Omega_\delta}(|{D\boldsymbol{u}}|-1)_+^{p-2-\alpha}\left|\frac{\Delta_{\infty}\boldsymbol u}{|D\boldsymbol{u}|^2}\right|^2|D\boldsymbol{u}|\,dx\leq C.
\end{equation}
Finally, since the constant $C$ does not depend on $\delta$, we can let \(\delta\to 0\). The  sets $\Omega_\delta$ invade the whole non-degenerate region $\tilde\Omega \cap \{|D\boldsymbol{u}|>1\}$. By the monotone convergence theorem, we obtain the  conclusion \eqref{eq:stimaseconde}. We remark that, even if $\alpha<0$, estimate~\eqref{eq:stimaseconde} still holds, since $D\boldsymbol{u}$ is locally bounded.

\end{proof}

\section{Integrability of the inverse of the weight}

In this section, we establish an integrability result for the inverse of the weight, and we give the proof of Theorem~\ref{TeoinversodelPeso}.

\begin{lemma}\label{lemmainversopeso}
Let \(\Omega\) be an open set in \(\mathbb{R}^n\) and assume that \(B_{R}\subset\subset \Omega\). Let \(\boldsymbol{u}_\varepsilon\) be a weak solution of \eqref{eq:problregol} and let \(\psi\) be the function defined in \eqref{eq:varphi}. Moreover, we assume $\boldsymbol{f}\in L^q_{loc}(\Omega)$, with $q>n$. Then, for \(k\) fixed in \(\{1,\dots,N\}\) and for parameters \(\theta,\delta,\beta>0\), the following inequality holds:
    \begin{equation*}
    \begin{split}
        \int_{B_{R}} {f}^k\frac{{\psi}^2}{(\delta+(|{D\boldsymbol{u}_\varepsilon}|-1)_+)^\beta} \,dx &\leq 2\theta \beta\int_{B_{R}}\frac{{\psi}^2 }{(\delta+(|{D\boldsymbol{u}_\varepsilon}|-1)_+)^{\beta}}\,dx
        \\&+ 2 \int_{B_{R}} \varepsilon\delta^{-\beta} \psi |\nabla {u}^k_\varepsilon||\nabla \psi| \,dx
        \\&+\frac{\beta}{4\theta}\int_{B_{R}}\varepsilon^2 |D\boldsymbol {u}_\varepsilon|^2\Big|\frac{{D^2\boldsymbol u_\varepsilon}{D\boldsymbol{u}_\varepsilon}}{|{D\boldsymbol{u}_\varepsilon}|}\Big|^2{\psi}^2 \delta^{-\beta-2}\chi_{0,\varepsilon}\,dx
        \\&+ 2 \int_{B_{R}}(|D\boldsymbol{u}_\varepsilon|-1)_+^{p-1-\beta} \psi |\nabla\psi| \,dx
        \\&+\frac{\beta}{4\theta}\int_{B_{R}}(|D\boldsymbol{u}_\varepsilon|-1)_+^{2(p-1)-\beta-2}\left| \frac{\Delta_{\infty}\boldsymbol{u}_\varepsilon}{|{D\boldsymbol{u}_\varepsilon}|^2} \right|^2{\psi}^2 \chi_{0,\varepsilon}\,dx.
\end{split}
\end{equation*}
\end{lemma}

\begin{proof}
Let $\beta \in \R$ be positive and define the following function $\boldsymbol{\zeta}:=(0,\dots,\varphi^k,\dots,0)$, with 
\begin{equation*}
    \varphi^k := \frac{{\psi}^2}{(\delta+(|{D\boldsymbol{u}_\varepsilon}|-1)_+)^\beta},
\end{equation*} 
where \(\psi\) has been defined in \eqref{eq:varphi} and $\delta$ is a positive constant. So 
\begin{equation}\label{gradientedellatestnelpeso}
    \nabla \varphi^k = 2\frac{{\psi} \nabla \psi}{(\delta+(|{D\boldsymbol{u}_\varepsilon}|-1)_+)^\beta}-\beta\frac{{\psi}^2 }{(\delta+(|{D\boldsymbol{u}_\varepsilon}|-1)_+)^{\beta+1}}\frac{{D^2\boldsymbol u_\varepsilon}{D\boldsymbol{u}_\varepsilon}}{|{D\boldsymbol{u}_\varepsilon}|}\chi_{0,\varepsilon},
\end{equation}
where \(\chi_{0,\varepsilon}\) is defined below \eqref{gradientebombolo}. Using $\boldsymbol{\zeta}$ as test function in \eqref{eq:deboleregolar}, it follows that
\begin{equation}\label{cosette}
    \begin{split}
         & \int_{B_{R}} \varepsilon \langle\nabla {u}^k_\varepsilon,\nabla \varphi^k\rangle+(|D\boldsymbol{u}_\varepsilon|-1)_+^{p-1}\Big\langle\frac{\nabla{u}^k_\varepsilon}{|D\boldsymbol{u}_\varepsilon|},\nabla \varphi^k \Big\rangle\,dx = \int_{B_{R}} {f}^k\varphi^k \,dx.
\end{split}
\end{equation}

Plugging \eqref{gradientedellatestnelpeso} in \eqref{cosette}, we infer that the l.h.s. of \eqref{cosette} equals
\begin{equation}\label{terminazzo}
    \begin{split}
       & \int_{B_{R}} \varepsilon \langle\nabla {u}^k_\varepsilon,\nabla \psi \rangle \frac{2\psi}{(\delta+(|{D\boldsymbol{u}_\varepsilon}|-1)_+)^\beta}\,dx \\
        &\quad -\beta\int_{B_{R}}\varepsilon \Big\langle\nabla {u}^k_\varepsilon,\frac{{D^2\boldsymbol u_\varepsilon}{D\boldsymbol{u}_\varepsilon}}{|{D\boldsymbol{u}_\varepsilon}|}\Big\rangle\frac{{\psi}^2\chi_{0,\varepsilon} }{(\delta+(|{D\boldsymbol{u}_\varepsilon}|-1)_+)^{\beta+1}}\,dx \\
        &\quad +\int_{B_{R}}(|D\boldsymbol{u}_\varepsilon|-1)_+^{p-1}\Big\langle\frac{\nabla{u}^k_\varepsilon}{|D\boldsymbol{u}_\varepsilon|},\nabla\psi\Big\rangle \frac{2\psi}{(\delta+(|{D\boldsymbol{u}_\varepsilon}|-1)_+)^\beta}\,dx \\
        &\quad -\beta\int_{B_{R}}(|D\boldsymbol{u}_\varepsilon|-1)_+^{p-1}\Big\langle\frac{\nabla{u}^k_\varepsilon}{|D\boldsymbol{u}_\varepsilon|},\frac{{D^2\boldsymbol u_\varepsilon}{D\boldsymbol{u}_\varepsilon}}{|{D\boldsymbol{u}_\varepsilon}|} \Big\rangle\frac{{\psi}^2\chi_{0,\varepsilon} }{(\delta+(|{D\boldsymbol{u}_\varepsilon}|-1)_+)^{\beta+1}}\,dx \\
        &=: \mathcal{I}_1 + \mathcal{I}_2 + \mathcal{I}_3 + \mathcal{I}_4.
\end{split}
\end{equation}

We begin by estimating the term $\mathcal{I}_1$. Using Cauchy-Schwarz inequality and since $\delta+(|D\boldsymbol{u}_\varepsilon|-1)_+\geq \delta$, we get 
\begin{equation}\label{firstpezzo}
    \begin{split}
        \mathcal{I}_1 &\leq \int_{B_{R}} \varepsilon |\nabla {u}^k_\varepsilon||\nabla \psi| \frac{2\psi}{(\delta+(|{D\boldsymbol{u}_\varepsilon}|-1)_+)^\beta}\,dx \\
        &\leq 2 \int_{B_{R}} \varepsilon\delta^{-\beta} \psi |\nabla {u}^k_\varepsilon||\nabla \psi| \,dx.
\end{split}
\end{equation}

For the term $\mathcal{I}_2$, using the weighted Young's inequality and since $\delta+(|D\boldsymbol{u}_\varepsilon|-1)_+\geq \delta$, we get
\begin{equation}\label{secondopezzo}
    \begin{split}
        \mathcal{I}_2 &\leq \theta \beta\int_{B_{R}}\frac{{\psi}^2 }{(\delta+(|{D\boldsymbol{u}_\varepsilon}|-1)_+)^{\beta}}\,dx \\
        &\quad +\frac{\beta}{4\theta}\int_{B_{R}}\varepsilon^2 |\nabla {u}^k_\varepsilon|^2\Big|\frac{{D^2\boldsymbol u_\varepsilon}{D\boldsymbol{u}_\varepsilon}}{|{D\boldsymbol{u}_\varepsilon}|}\Big|^2\frac{{\psi}^2 \chi_{0,\varepsilon}}{(\delta+(|{D\boldsymbol{u}_\varepsilon}|-1)_+)^{\beta+2}}\,dx \\
        &\leq \theta \beta\int_{B_{R}}\frac{{\psi}^2 }{(\delta+(|{D\boldsymbol{u}_\varepsilon}|-1)_+)^{\beta}}\,dx \\
        &\quad +\frac{\beta}{4\theta}\int_{B_{R}}\varepsilon^2 |D\boldsymbol {u}_\varepsilon|^2\Big|\frac{{D^2\boldsymbol u_\varepsilon}{D\boldsymbol{u}_\varepsilon}}{|{D\boldsymbol{u}_\varepsilon}|}\Big|^2{\psi}^2 \delta^{-\beta-2}\chi_{0,\varepsilon}\,dx.
\end{split}
\end{equation}

For the term $\mathcal{I}_3$, by Cauchy-Schwarz's inequality and since $\delta+(|D\boldsymbol{u}_\varepsilon|-1)_+\geq (|D\boldsymbol{u}_\varepsilon|-1)_+$, it follows that
\begin{equation}\label{terzopezzo}
    \begin{split}
        \mathcal{I}_3 \leq 2 \int_{B_{R}}(|D\boldsymbol{u}_\varepsilon|-1)_+^{p-1-\beta} \psi |\nabla\psi| \,dx. 
    \end{split}
\end{equation}

For the last term, using the weighted Young's inequality and since $\delta+(|D\boldsymbol{u}_\varepsilon|-1)_+\geq (|D\boldsymbol{u}_\varepsilon|-1)_+$, we have 
\begin{equation}\label{quartopezzo}
    \begin{split}
        \mathcal{I}_4 &\leq \theta \beta\int_{B_{R}}\frac{{\psi}^2 }{(\delta+(|{D\boldsymbol{u}_\varepsilon}|-1)_+)^{\beta}}\,dx \\
        &\quad +\frac{\beta}{4\theta}\int_{B_{R}}(|D\boldsymbol{u}_\varepsilon|-1)_+^{2(p-1)-\beta-2}\left| \frac{\Delta_{\infty}\boldsymbol{u}_\varepsilon}{|{D\boldsymbol{u}_\varepsilon}|^2} \right|^2{\psi}^2 \chi_{0,\varepsilon}\,dx,
\end{split}
\end{equation}
where in the last inequality we have used that $\langle \nabla u_\varepsilon^k,D^2 \boldsymbol{u}_\varepsilon D\boldsymbol{u}_\varepsilon\rangle$ is the $k^{th}$ component of $\Delta_\infty\boldsymbol{u}_\varepsilon$.
Collecting \eqref{terminazzo}, \eqref{firstpezzo}, \eqref{secondopezzo}, \eqref{terzopezzo} and \eqref{quartopezzo} into \eqref{cosette}, we get the thesis.

\end{proof}

Now we are in position to prove Theorem \ref{TeoinversodelPeso}.
\begin{proof}[Proof of Theorem \ref{TeoinversodelPeso}]
Let $ B_{R}\subset\subset\Omega$ and let $\boldsymbol{u}_\varepsilon$ be a weak solution to the regularized problem \eqref{eq:problregol}. By assumptions on $\boldsymbol{f}$, by Lemma~\ref{lemmainversopeso}, and for \(\theta\) sufficiently small, we deduce that
\begin{equation}\label{robadaestimare}
    \begin{split}
        C(\theta,\beta,\tau)&\int_{B_{R}} \frac{{\psi}^2}{(\delta+(|{D\boldsymbol{u}_\varepsilon}|-1)_+)^\beta} \,dx \leq 2 \int_{B_{R}} \varepsilon\delta^{-\beta} \psi |\nabla {u}^k_\varepsilon||\nabla \psi| \,dx \\
        &\quad +\frac{\beta}{4\theta}\int_{B_{R}}\varepsilon^2 |D\boldsymbol {u}_\varepsilon|^2\Big|\frac{{D^2\boldsymbol u_\varepsilon}{D\boldsymbol{u}_\varepsilon}}{|{D\boldsymbol{u}_\varepsilon}|}\Big|^2{\psi}^2 \delta^{-\beta-2}\chi_{0,\varepsilon}\,dx \\
        &\quad + 2 \int_{B_{R}}(|D\boldsymbol{u}_\varepsilon|-1)_+^{p-1-\beta} \psi |\nabla\psi| \,dx \\
        &\quad +\frac{\beta}{4\theta}\int_{B_{R}}(|D\boldsymbol{u}_\varepsilon|-1)_+^{2(p-1)-\beta-2}\left| \frac{\Delta_{\infty}\boldsymbol{u}_\varepsilon}{|{D\boldsymbol{u}_\varepsilon}|^2} \right|^2{\psi}^2  \chi_{0,\varepsilon}\,dx \\
        &=: \mathcal{I}_1 + \mathcal{I}_2 + \mathcal{I}_3 + \mathcal{I}_4,
    \end{split}
\end{equation}
where \(C(\theta,\beta,\tau)\) is a positive constant, \(k\in\{1,\dots,N\}\), \(\delta,\beta>0\), and \(\psi\) is the function defined in \eqref{eq:varphi}. Our aim is now to obtain an estimate for the right-hand side of \eqref{robadaestimare} which is uniform with respect to \(\varepsilon\). To this end, we choose \(\beta\) as in the statement of Theorem \ref{TeoinversodelPeso} and set \(\delta:=\varepsilon^{\frac{1}{\beta+2}}\). Thanks to this choice, and by \eqref{stimadelgradiente}, it follows that
\begin{equation}\label{stima100}
    \mathcal{I}_1+\mathcal{I}_3\leq C(n,N,p,R,\|\boldsymbol{f}\|_{L^q},\|D\boldsymbol{u}\|_{L^p}).
\end{equation}

Now we estimate the term $\mathcal{I}_2$. We fix \(\alpha\) as in Proposition~\ref{bellazio}. By Proposition~\ref{bellazio} (which provides an upper bound for the first term on the left-hand side of \eqref{bombolone}) and by \eqref{stimadelgradiente}, we obtain \((|D\boldsymbol u_\varepsilon|-1)_+^{\alpha}\,|D\boldsymbol u_\varepsilon|\leq C\), so that 
\begin{equation}\label{stima101}
    \begin{split}
        \mathcal{I}_{2} &= \frac{\beta}{4\theta}\int_{B_{R}}\varepsilon^2 (|D\boldsymbol u_\varepsilon|-1)_+^{-\alpha}(|D\boldsymbol u_\varepsilon|-1)_+^{\alpha}|D\boldsymbol{u}_\varepsilon|^2\Big|\frac{{D^2\boldsymbol u_\varepsilon}{D\boldsymbol{u}_\varepsilon}}{|{D\boldsymbol{u}_\varepsilon}|}\Big|^2{\psi}^2 \delta^{-\beta-2}\chi_{0,\varepsilon}\,dx \\
        &\leq C \int_{B_{R}}\varepsilon (|D\boldsymbol u_\varepsilon|-1)_+^{-\alpha}|D\boldsymbol{u}_\varepsilon|\Big|\frac{{D^2\boldsymbol u_\varepsilon}{D\boldsymbol{u}_\varepsilon}}{|{D\boldsymbol{u}_\varepsilon}|}\Big|^2{\psi}^2\chi_{0,\varepsilon} \,dx \leq C,
    \end{split}
\end{equation}
where $C$ is a positive constant not depending on $\varepsilon$. For the term $\mathcal{I}_4$, by assumptions on $\beta$ and by \eqref{stimauniforme} we get
\begin{equation}\label{stima103}
    \mathcal{I}_4\leq C(p,n,N,\beta,R,\|\boldsymbol{f}\|_{L^q},\|\boldsymbol{f}\|_{W^{1,1}},\|D\boldsymbol{u}\|_{L^p}).
\end{equation}

By \eqref{stima100}, \eqref{stima101} and \eqref{stima103}, once these bounds are inserted into \eqref{robadaestimare}, we obtain
\begin{equation}\label{bombazzo}
   \int_{B_{\tilde R}} \frac{1}{(\varepsilon^{\frac{1}{\beta+2}}+(|{D\boldsymbol{u}_\varepsilon}|-1)_+)^\beta} \,dx\leq C, 
\end{equation}
where $C$ is a constant not depending on $\varepsilon$. 
We consider $A:=\{x\in B_{\tilde R}\,:\, |D\boldsymbol{u}|\leq1\}$. By \eqref{bombazzo} and Fatou's Lemma we get 
\begin{equation*}\label{bombazzo2}
\begin{split}
   \int_{A} \liminf_{\varepsilon\rightarrow 0}\frac{1}{(\varepsilon^{\frac{1}{\beta+2}}+(|{D\boldsymbol{u}_\varepsilon}|-1)_+)^\beta} \,d\leq\liminf_{\varepsilon\rightarrow 0}\int_{A}  \frac{1}{(\varepsilon^{\frac{1}{\beta+2}}+(|{D\boldsymbol{u}_\varepsilon}|-1)_+)^\beta} \,dx\leq C.
   \end{split}
\end{equation*}
On the other hand we know from \cite[Lemma 3.3]{BDGP1} that $(|D\boldsymbol{u}_\varepsilon|-1-\delta)_+$ converges to $ (|D\boldsymbol{u}|-1-\delta)_+$, a.e. on $B_{\tilde R}$, for any $\delta>0$. From this fact one can deduce that 
$\limsup_{\varepsilon\rightarrow 0}{(\varepsilon^{\frac{1}{\beta+2}}+(|{D\boldsymbol{u}_\varepsilon}|-1)_+)^\beta}=0$. Therefore $\mathcal{L}(A)=0$. The estimate \eqref{eq:stimainverso} follows.
\end{proof}

We conclude with the proof of Theorem~\ref{Corollario2}.

\begin{proof}[Proof of Theorem \ref{Corollario2}]
\noindent The estimate~\eqref{corollario} is obtained from  ~\eqref{eq:stimaseconde} upon choosing 
$\alpha = p-2$ for $p<3$, and using the fact that $D\mathbf{u}$ is locally bounded.

If $p\geq 3$, we split ${\left|\Delta_{\infty}\boldsymbol u\right|}^q$ as follows:
    \begin{equation*}
        {\left|\Delta_{\infty}\boldsymbol u\right|}^q=(|D \boldsymbol{u}|-1)_+^{\frac{(p-2-\alpha)q}{2}} {\left|\Delta_{\infty}\boldsymbol u\right|}^q (|D \boldsymbol{u}|-1)_+^{\frac{(\alpha+2-p)q}{2}}
    \end{equation*}
    Integrating over $\Omega_\delta := \{ x \in \tilde{\Omega} : |D\boldsymbol{u}(x)| > 1+\delta \}$, for $\delta>0$, using the estimate \eqref{eq:stimaseconde}, and by H\"older inequality (note that that if $p\geq3$, then $q<2$), we deduce
    \begin{equation*}
    \begin{split}
        \int_{\Omega_\delta} &{\left|\Delta_{\infty}\boldsymbol u\right|}^q\,dx = \int_{\Omega_\delta}(|D \boldsymbol{u}|-1)_+^{\frac{(p-2-\alpha)q}{2}} {\left|\Delta_{\infty}\boldsymbol u\right|}^q (|D \boldsymbol{u}|-1)_+^{\frac{(\alpha+2-p)q}{2}}\,dx\\
        &\leq \left( \int_{\Omega_\delta}(|D \boldsymbol{u}|-1)_+^{
        p-2-\alpha} {\left|\Delta_{\infty}\boldsymbol u\right|}^2 \,dx \right)^{\frac{q}{2}} \left ( \int_{\Omega_\delta}(|D \boldsymbol{u}|-1)_+^{\frac{(\alpha+2-p)q}{2-q}}\,dx\right)^{\frac{2-q}{2}}\\
        & \leq C\left ( \int_{\Omega_\delta}(|D \boldsymbol{u}|-1)_+^{\frac{(\alpha+2-p)q}{2-q}}\,dx\right)^{\frac{2-q}{2}}
    \end{split}
    \end{equation*}
    We notice that since $p \geq 3$, taking $\alpha \simeq 1$, if $q < \frac{p-1}{p-2}$ the right hand side is bounded. Indeed, it follows by Theorem \ref{TeoinversodelPeso}, since $\frac{(p-2-\alpha)q}{2-q} < p-1$.\\
    Therefore,
    \begin{equation*}
        \int_{\Omega_\delta} {\left|\Delta_{\infty}\boldsymbol u\right|}^q\,dx \leq C,
    \end{equation*}
where $C$ does not depend on $\delta$. Letting $\delta \to 0$, using the fact that the Lebesgue measure of $\{|D\mathbf{u}|\le 1\}$ is zero, we obtain the thesis.

\end{proof}

\end{document}